\documentclass[11pt]{article}
\usepackage{amsmath, amssymb, amscd, amsthm, amsfonts}
\usepackage{graphicx}
\usepackage{hyperref}
\usepackage{mathtools}
\usepackage{mathrsfs}
\usepackage{enumerate}
\usepackage{mathabx}
\usepackage{setspace}
\usepackage{color}
\usepackage{comment}
\usepackage{tikz}
\usepackage{tikz-cd}

\theoremstyle{plain}
\newtheorem{thm}{Theorem}[section]
\newtheorem{prop}[thm]{Proposition}
\newtheorem{lemm}[thm]{Lemma}
\newtheorem{cor}[thm]{Corollary}

\theoremstyle{definition}
\newtheorem{defn}[thm]{Definition}

\newtheorem{rmk}[thm]{Remark}
\newtheorem{con}[thm]{Conjecture}

\newcommand{\cat}[1]{\text{\rm cat}\left(#1 \right)}

\title{Surgery Approach to Rudyak's Conjecture}
\author{Alexander Dranishnikov and Jamie Scott}

\begin{document}
\maketitle
\begin{abstract}
Using the surgery we prove the following:  
\begin{thm}
Let $f:M \to N$ be a normal map of degree one between closed smooth manifolds 
with $N$ being $(r-1)$-connected, $r\ge 1$. If $N$ satisfies the inequality
$\dim N \leq 2r \cat N - 3$, 
then {\rm $\cat M \geq \cat N$}.
\end{thm}
\end{abstract}

\section{Introduction}
\begin{defn}
The (reduced) Lusternik-Schnirelmann category, $\cat X$, of an ANR space $X$ is the minimal number $k$ such that $X$ admits an open cover by $k+1$ sets $U_0,U_1,\dots,U_k$ such that each $U_i$ is contractible in $X$.
\end{defn}
The Lusternik-Schnirelmann category is an important invariant, since it gives a lower bound on the number of critical points for a smooth real-valued function on a closed manifold~\cite{LS}. In this paper we restrict ourselves to smooth (or PL) manifolds, in particular, to the following conjecture~\cite{Ru1}:
\begin{con}[Rudyak's Conjecture]
If $f:M\to N$ is a degree one map between closed manifolds, then
$\cat M\ge\cat N$.
\end{con}
A simple example of a degree one map is the collapsing map for a connected sum onto one of its summands
$q:M\# N\to N$. Even for this special case, Rudyak's Conjecture was proven only recently~\cite{DS1}. In~\cite{Ru1}, Rudyak proved his conjecture in the following special case. 

\begin{thm}[Rudyak's Theorem]
Let $f:M\to N$ be a map of degree one between  closed stably parallelizable PL manifolds with $N$ being  $(r-1)$-connected for some $r\ge 1$. If $N$ satisfies the inequality
$\dim N\le 2r\cat N-4$, then $\cat M\ge\cat N$.
\end{thm}
Rudyak's conjecture was also proven for manifolds of dimension $\le 4$ (see~\cite{Ru2}).

Surgery theory is a cornerstone in the classification problem for closed manifolds. A major step in this theory is turning a degree one normal map between manifolds into a homotopy equivalence.
This paper explores the applications of this technique from surgery theory to Rudyak's Conjecture. 

Our main result is Theorem 0.1. We note that it holds true for PL manifolds as well.
Though our approach differs from Rudyak's, surprisingly we obtained almost the same restrictions in our theorem on the dimension and the category as in~\cite{Ru1}, $\dim N\le 2r\cat N-4$. Our restriction is only slightly looser: $\dim N\le 2r\cat N-3$.
We recall that Rudyak's approach is based on stable cohomotopy theory, while we use surgery theory. Thus, our theorem replaces the stably parallelizable assumption on the manifolds 
in Rudyak's theorem by the weaker condition that the map $f$ is normal, and loosens the bound on dimension by one.

\section{Preliminaries}

\subsection{Connected sum formula}
We recall that the connected sum $M\# N$ of two manifolds depends on orientations whenever $M$ and $N$ are orientable. For an orientable manifold $M$ and $k\in\mathbb Z\setminus\{0\}$, by $kM$ we denote the connected sum of $k$ copies of $M$ if $k>0$ and the connected sum of $|k|$ copies of $\bar M$ if $k<0$, where $\bar M$ denotes $M$ with the opposite orientation.
\begin{thm}[\cite{DS1,DS2}]\label{sum}
For any two closed manifolds $$\cat{M\# N}=\max\{\cat M,\cat N\}.$$
\end{thm}

\subsection{Lower bound for cat} Clearly, for a CW complex $X$, the dimension $\dim X$ is an upper bound for $\cat X$. A natural lower bound for the LS-category is the cup length.
The {\it cup length} of a space $X$, $\smile$-length$(X)$, is defined as the maximal $k$ such that there are $k$ cohomology classes of dimension $>0$, $\alpha_i\in H^*(X;M_i)$, $i=1,\dots,k$ with the cup product $\alpha_1\smile\dots\smile\alpha_k\ne0$.
Here the cup product is understood in a generalized sense: All $M_i$ are $\pi_1(X)$-modules
and $\alpha_1\smile\dots\smile\alpha_k\in H^*(X;M_1\otimes\dots\otimes M_k)$ where the tensor product is taken over $\mathbb Z$. The following theorem is classic (we refer to~\cite{CLOT} for a proof):
\begin{thm}
If $X$ is an ANR, then $\smile$-length$(X)\le\cat X$.
\end{thm}

\subsection{Essential manifolds} In ~\cite{G}, Gromov calls a closed $n$-manifold $M$ {\em essential} if
a map $u_M:M\to B\pi$ that induces an isomorphism of the fundamental groups, $\pi=\pi_1(M)$, cannot be deformed into the $(n-1)$-skeleton $B\pi^{(n-1)}$. The following is well-known:
\begin{thm}[\cite{Sch,DR}]
For the cohomological dimension of a discrete group $\pi$,
$$cd(\pi)=\max\{k\mid\beta_\pi^k\ne 0\}$$ where $\beta_\pi\in H^1(B\pi;I(\pi))$ is the Berstein-Schwarz class.
\end{thm}
This theorem implies that for an essential $n$-manifold $M$, $(u^*\beta_\pi)^n\ne 0$, i.e.,
$\smile$-length$(M)=n$, and, hence, $\cat M=n$. We note that the converse also holds true: The condition $\cat M=n$ for a closed $n$-manifold implies that
$M$ is essential~\cite{KR}.

\

\subsection{Ganea-Schwarz's approach to cat}Recall that an element of an iterated join $X_0*X_1*\cdots*X_n$ of topological spaces is a formal linear combination $t_0x_0+\cdots +t_nx_n$ of points $x_i\in X_i$ with $\sum t_i=1$, $t_i\ge 0$, in which all terms of the form $0x_i$ are dropped. Given fibrations $f_i:X_i\to Y$ for $i=0, ..., n$, the fiberwise join of spaces $X_0, ..., X_n$ is defined to be the space
\[
    X_0*_YX_1*_Y\cdots *_YX_n=\{\ t_0x_0+\cdots +t_nx_n\in X_0*\cdots *X_n\ |\ f_0(x_0)=\cdots =f_n(x_n)\ \}.
\]
The fiberwise join of fibrations $f_0, ..., f_n$ is the fibration 
\[
    f_0*_Y*\cdots *_Yf_n: X_0*_YX_1*_Y\cdots *_YX_n \longrightarrow Y
\]
defined by taking a point $t_0x_0+\cdots +t_nx_n$ to $f_i(x_i)$ for any $i$ such that $t_i \neq 0$. 

When $X_i=X$ and $f_i=f:X\to Y$ for all $i$  the fiberwise join of spaces is denoted by $*^{n+1}_YX$ and the fiberwise join of fibrations is denoted by $*_Y^{n+1}f$. 

For a path connected space $X$, we turn an inclusion of a point $*\to X$ into a fibration $p^X_0:G_0(X)\to X$, whose fiber is known to be the loop space $\Omega X$.  The $n$-th Ganea space of $X$ is defined to be the space $G_n(X)=*_X^{n+1}G_0(X)$, while the $n$-th Ganea fibration $p^X_n:G_n(X)\to X$
is the fiberwise join $\ast^{n+1}_Xp^X_0$. Then the fiber of
$p^X_n$ is $\ast^{n+1}\Omega X$.

The Schwarz theorem~\cite{Sch},\cite{CLOT} implies the following: 
\begin{thm}
If $X$ is a connected ANR, then  $\cat X\le n$ if and only if the fibration $p^X_n:G_n(X)\to X$ admits a section.
\end{thm}

\subsection{Maps of degree one} 
A typical example of a degree one map is the collapsing of the connected sum
$q:M\# N\to N$ onto a summand. So it is natural to think that 
the codomain of a degree one map is somehow less ``complex" than the domain. For general degree one maps this is supported by the following well-known fact:

\begin{prop} \label{epihom}
If $f:M \to N$ is a degree one map between closed manifolds, then for any coefficients the induced homomorphism $f_\ast:H_\ast(M) \to H_\ast(N)$ is a split surjection and $f^*:H^\ast(N)\to H^\ast(M)$ is a split injection.
\end{prop}
We refer for the proof to \cite{Ru1} or \cite{DS1}.

\subsection{The trace of the surgery} Suppose that $S\subset M$ is an embedded $k$-sphere $S^k$ into an $n$-manifold $M$ that has a trivial tubular neighborhood
$N=S^k\times D^{n-k}$. Thus, $N\subset\partial(D^{k+1}\times D^{n-k})$. The   $n$-manifold $M'=M\setminus Int(N)\cup (D^{k+1}\times\partial D^{n-k})$ 
is said to be obtained from $M$ by {\em the $k$-surgery along the sphere $S$}. The {\em trace of this surgery} is defined to be the bordism $W$ between $M$ and $M'$ defined as $(M\times[0,1])\cup (D^{k+1}\times D^{n-k})$.
\begin{prop}
The trace $W$ of a $k$-surgery along $S$ is homotopy equivalent to $M\cup_\phi D^{k+1}$ for some attaching map $\phi$. Moreover, there is a homotopy equivalence of pairs
$h:(M\cup_\phi D^{k+1},M)\to (W,M)$.
\end{prop}
\begin{proof}
We define a homotopy equivalence  $g:W\to M\cup_\phi D^{k+1}$.
We consider a cell-like map $q:W\to X$ defined by the decomposition of $W$ into
the $(n-k)$-disks $x\times D^{n-k}$, $x\in D^{k+1}$, and singletons. Then  
$q(M)\cong M$ and $q(M\times [0,1])\cong M\times[0,1]$. Thus, we obtain a homotopy equivalence $q:W \to q(M \times [0,1]) \cup q(D^{k+1} \times D^{n-k}) = (M\times[0,1])\cup_\psi D^{k+1}$, where $\psi$ is the composition of the attaching map of the handle $D^{k+1} \times D^{n-k}$ with $q$.
We extend the projection $\pi:M\times[0,1]\to M$ to a 
homotopy equivalence  $g:W\to M\cup_{\pi\circ\psi} D^{k+1}$. Clearly, the restriction of $g$ to $M$ is a homotopy equivalence.
\end{proof}
The trace of the surgery is defined for a chain of consecutive surgeries by gluing together
the corresponding bordisms so that the argument of the above proposition can be applied to a chain of surgeries.
\begin{cor}\label{he}
There is a homotopy equivalence of pairs $$h:(M\cup_{\phi_1}D^{n_1}\cup_{\phi_2}D^{n_2}\cup\dots\cup_{\phi_k}D^{n_k},M)\to (W,M)$$ where  $W$ is the trace of a chain of  surgeries in dimensions $n_1-1,\dots,n_k-1$ originating on a manifold $M$.
\end{cor}
\subsection{Surgery obstruction} Doing surgery is a way to kill the homotopy groups of a manifold. If we can do surgery on some generators for each of the homotopy groups up to the middle dimension, then we obtain a sphere due to the Poincare Conjecture (Theorem). In the case of a map $f:M\to N$ between manifolds, if we can do the same for the kernel of the induced homomorphism for the homotopy groups, then we will obtain a homotopy equivalence. For normal maps this can be done
up to one less than the middle dimension.

We recall that a map between manifolds $f:M\to N$ is called {\em normal} if there is a vector bundle $\nu$ on $N$ such that $\tau_M\oplus f^*\nu$ is a stably trivial bundle~\cite{B},
where $\tau_M$ denotes the tangent bundle of $M$. Note that every map of a stably parallelizable manifold is normal.
Another observation is that the connected sum $f\# g:M\# M'\to N\# N'$ of two normal maps is normal.

\begin{thm}[The main theorem of surgery~\cite{W}] For every group $\pi$ there is a 4-periodic sequence of abelian groups $L_n(\pi)$ such that in dimension $n\ge 5$ for any degree one normal map $f:M\to N$ between closed manifolds there is an element
$\theta(f)\in L_n(\pi_1(N))$, called {\rm the surgery obstruction}, such that $f$
is (normally) bordant by means of a sequence of surgeries in dimensions
$\le\dim M/2$ to a (simple) homotopy equivalence if and only if
the surgery obstruction $\theta(f)\in L_n(\pi_1(N))$ is trivial. 
\end{thm}

Let $\bar f:\bar M\to\bar N$ be  the normal degree one map $f:M\to N$ where the manifolds are taken with the opposite orientations. Then $\theta(\bar f)=-\theta(f)$. Also, we note that $\theta(f\#g)=\theta(f)+\theta(g)$ for any two normal maps of degree one.
For both facts, we refer
to the definition of the $L$-group as it presented in Chapter 9 of \cite{W}.

\section{Main Result}
\begin{lemm}\label{id}
If there is a degree one map $f:M\to N$ between smooth $n$-manifolds, then there is a degree one map $g:M'\to N$ which is a bijection $g^{-1}|_B:B\to g^{-1}(B)$ over an open $n$-ball $B\subset N$ such $\cat {M'}=\cat M$. 
\end{lemm}
\begin{proof}
First suppose $\cat M=1$ so that $M$ must be a sphere. Since $f$ is of degree $1$, $f_\ast$ is an epimorphism on integral homology by Proposition \ref{epihom}. Since $f_\ast$ is an isomorphism at $\dim M$ by assumption and since the homology groups of $M$ are trivial below $\dim M$, it follows that $f_\ast$ must be an isomorphism on integral homology. Therefore, $f$ must be a homotopy equivalence by Whitehead's Theorem so that $g$ can simply be defined to be $Id_N:N \to N$.

Now suppose $\cat M \geq 2$. We may assume that $f$ is a smooth map. Let $y\in N$ be a regular value, i.e.,
there is a closed neighborhood $D$ of $y$ homeomorphic to a ball such that $f^{-1}(D)=\coprod_{i=1}^m D_i$ and the restriction of $f$ to each $D_i$ is a homeomorphism onto $D$. Since the degree of $f$ is one, the sum of local degrees $\Sigma_i \deg(f|_{D_i})$ is one. Hence, $m=2k+1$. Assume that the local degree of $f$ at $D_m$ is one. Then the balls $D_i$ can be arranged into pairs $\pm D_j$, $j=1,\dots,k$ in a way that the local degree at $+D_j$ is 1 and the local degree at $-D_j$ is -1. For each such pair we remove $int(+D_j)\cup int(-D_j)$
from $M$ and attach a 1-handle $H_j=S^{n-1}\times[0,1]$ instead. The map $f$ can be extended to $H_j$ as the projection onto the first factor where $S^{n-1}$ is identified with $\partial D$. In this way we construct a manifold $M'$ and a degree one map $g:M'\to N$ such that $g$ is bijective over $D$. Since
$M'=M\#k(S^{n-1}\times S^1)$, we obtain
$\cat{M'}=\max\{\cat M,\cat{S^{n-1}\times S^1}\}=\max\{\cat M, 2\}=\cat M$ by the connected sum formula (Theorem~\ref{sum}).
\end{proof}

We note that the same works for PL-manifolds by approximating $f$ by a noncollapsing simplicial map.

\begin{lemm}\label{mainlemma}
Let $f:M \to N$ be a normal map of degree one between closed orientable $n$-manifolds, $n\ge 5$, with $N$ being $(r-1)$-connected for some $r \geq 1$. If $f$ has zero surgery obstruction and $N$ satisfies the inequality $\dim N \leq 2r \cat N - 3$, 
then {\rm $\cat M \geq \cat N$}.
\end{lemm}

\begin{proof} 
By way of contradiction assume that $\cat M \le q < \cat N=q+1$. 
Since $\cat f \le q$, there is a lift $\lambda:M\to G_q(N)$ such that the diagram
\[
\begin{tikzcd}
& G_{q}(N) \arrow[d, "p_{q}^N"]
\\ M \arrow[r, swap, "f"] \arrow[ur, "\lambda"]
& N
\end{tikzcd}
\]
commutes.
Since $f$ has zero surgery obstruction, there is some normal bordism $F:W \to N$ of $f$ to some homotopy equivalence $f':M' \to N$ where the manifold $W$ is the trace of surgeries in dimensions $\le\frac{\dim M}{2}$.
This gives us the following homotopy commutative diagram:
\[
\begin{tikzcd}
M \arrow[r,"\lambda"] \arrow[d, hook] & G_{q}(N) \arrow[d, "p_{q}^N"]
\\ W \arrow[r, swap, "F"] \arrow[ur, dashed, "L"]
& N
\end{tikzcd}
\]

Our goal is to find a lift $L$ of $F$, labeled above, that extends $\lambda$ to $W$.
In view of Corollary~\ref{he}, there is a homotopy equivalence of pairs
$$h:(M \cup_{\varphi_1}D^{n_1} \cup_{\varphi_2} ... \cup_{\varphi_k}D^{n_k},M)\to (W,M)$$
with  each $n_i \leq \frac{\dim M}{2} + 1$ when $\dim M$ is even and
with $n_i\leq\frac{\dim M-1}{2}+1$ when $\dim M$ is odd. So,  it suffices to solve the lifting problem 
\[
\begin{tikzcd}
W_{i-1} \arrow[r,"L_{i-1}"] \arrow[d, hook] & G_{q}(N) \arrow[d, "p_{q}^N"]
\\ W_i \arrow[r, swap, "F_i"] \arrow[ur, dashed, "L_i"]
& N
\end{tikzcd}
\]
for each $i>0$ where $W_i=M \cup_{\varphi_1}D^{n_1} \cup_{\varphi_2} ... \cup_{\varphi_i}D^{n_i}$, $W_0=M$, $L_0=\lambda$, and $F_i$ is the restriction of $F\circ h$ to $W_i$. 

We recall that the fiber of $p_{q}^N$ is $\ast^{q+1} \Omega(N)$. Since $\Omega N$ is $(r-2)$-connected, the fiber is
$(r(q+1)-2)$-connected. Thus,  $L_{i-1}$ can be extended to $D^{n_i}$ if $n_i - 1 \leq r(q+1) - 2$. Suppose that $n$ is even. Then $n$ and $2r(q+1)-3$ differ by at least $1$ so that our assumption strengthens to $n \leq 2r(q+1)-4$ and the lift $L_i$ exists since $n_i - 1 \leq \frac{n}{2}$. If $n$ is odd, then $n-1 \leq 2(r(q+1)-2)$ by the assumption and again the lift $L_i$ still exists. 

Thus, lift $L$ exists.
Now since $f'$ is a homotopy equivalence, it has some homotopy inverse $g'$ so that $L \circ g'$ is a homotopy section of $p^N_{q}$, but this contradicts the inequality $\cat N > q$.  Hence, $\cat M \geq \cat N$.
\end{proof}

\begin{thm}\label{thm}
Let $f:M \to N$ be a normal map of degree one between closed orientable smooth manifolds with $N$ being $(r-1)$-connected for some $r \geq 1$ and with $\dim N \leq 2r \cat N - 3$.
Then $\cat M \geq \cat N$.
\end{thm}
\begin{proof}
Since Rudyak's conjecture is known for 4-manifolds~\cite{Ru2}, we may assume that $\dim N >4$.
Let $\bar f:\bar M\to\bar N$ be as defined in subsection 2.7. Then we have $\theta(\bar f)=-\theta(f)$. By Lemma~\ref{id} we may assume that there is an open ball $B\subset N$ such that $f^{-1}$ is bijective over $B$.
Using $B$ and $f^{-1}(B)$ for the docking in the connected sum construction we consider 
the connected sum of maps $f\#\bar f:M\#\bar M\to N\#\bar N$.
Since the connected sum of normal maps is a normal map and the surgery is additive (see subsection 2.7), we obtain
$\theta(f\#\bar f)=0$. By Lemma~\ref{mainlemma}, $\cat{M\#\bar M}\ge\cat{N\#\bar N} $.
By the connected sum formula (Theorem~\ref{sum}), $\cat M\ge\cat N$.
\end{proof}

\begin{cor}
Let $f:M \to N$ be a normal map between closed orientable $n$-manifolds.  Then for any essential manifold $V$ of dimension $k\ge n-1$,
$$\cat{M\times V} \geq \cat {N\times V}.$$
\end{cor}
\begin{proof}
We assume that $\cat N>1$. This implies that $\smile$-length$ (N)\ge 2$.
Since $V$ is essential, it follows that the cup length of $V$ equals $k$ (see subsection 2.3).
Therefore, by the cup-length estimate, $\cat {N\times V}\ge k+2 $. Then
the condition of Theorem~\ref{thm} is satisfied: 
$$\dim(N\times V)=n+k\le 2k + 1 \le 2\cat {N\times V}-3.$$

If $\cat N=1$, then $\cat{M\times V}\ge\smile$-length$(M\times V)\ge k+1=\cat N+\cat V\ge\cat{N\times V}$.
\end{proof}

\begin{cor}\label{cor2}
Let $f:M \to N$ be a degree one normal map with  $(r-1)$-connected $N$ such that $\dim N \leq 6r-3$. Then $\cat M \geq \cat N$.
\end{cor}

\begin{proof}
If $\cat M = 1$, then $f$ is a homotopy equivalence (see the proof of Lemma \ref{id}) so that $\cat N = 2$ implies $\cat M \geq 2$. Thus, we may assume $\cat N \geq 3$. Then
\[ \dim N
\leq 6r-3
= 2r(3) - 3
\leq 2r \cat N - 3\]
so that $\cat M \geq \cat N$ by Theorem \ref{thm}.
\end{proof}
We note that the conditions of Corollary~\ref{cor2} are satisfied when the dimension of manifolds equals either $4r+k$ or $3r+k$ with $k<r$.

\begin{rmk} Since every map of a stably parallelizable manifold is normal, Theorem~\ref{thm} recovers Rudyak's Theorem. 
\end{rmk}

\

\section*{Acknowledgments}
The first author was supported by the Simons Foundation Grant.

\

\end{document}